\documentclass[reqno,11pt]{article}

\usepackage{geometry}
\usepackage{indentfirst}
\usepackage{amsthm}
\usepackage[numbers,comma,sort&compress,numbers]{natbib}
\usepackage{amsmath,amsfonts,amssymb,color}
\numberwithin{equation}{section}
\usepackage{enumerate}
\usepackage{fixltx2e}
\usepackage{mathtools}
\usepackage[T1]{fontenc}
\usepackage[utf8]{inputenc}
\usepackage{authblk}
\allowdisplaybreaks[4]

\theoremstyle{plain}
\newtheorem{theorem}{Theorem}[section]

\newtheorem{lemma}[theorem]{Lemma}

\theoremstyle{definition}

\newtheorem{remark}[theorem]{Remark}

\title{\textbf{Nontrivial solution for Klein-Gordon equation coupled with Born-Infeld theory with critical growth}}

\author{, Lin Li and }
\author{Chuan-Min He, \ Lin Li\footnote{{\tt Corresponding author.}
		E-mail address: {\tt linli@ctbu.edu.cn \& lilin420@gmail.com} (L. Li).},\ \  and\ Shang-Jie Chen\footnote{Lin Li is supported by Research Fund of National Natural Science Foundation of China (No. 11861046), Chongqing Municipal Education Commission (No. KJQN20190081), Chongqing Technology and Business University(No. CTBUZDPTTD201909).}\\
	\footnotesize
	School of Mathematics and Statistics \& Chongqing Key Laboratory of Economic and Social Application Statistics,\\ \footnotesize Chongqing Technology and Business University,\\
	\footnotesize Chongqing 400067, China}

\date{}

\begin{document}
\maketitle

\begin{abstract}
\noindent   In this paper, we study the following system
\begin{eqnarray*}
	\left\{   \begin{array}{ll}
	-\Delta u + V(x)u-(2\omega+\phi)\phi u=\lambda f(u)+|u|^{4}u,   \ & \text{in} \  \mathbb{R}^{3}, \\
	\Delta \phi + \beta\Delta_4\phi  = 4\pi(\omega+\phi) u^{2},   \ & \text{in}\  \mathbb{R}^{3},\\
\end{array}
\right.	
\end{eqnarray*}
where $f(u)$ without any growth and Ambrosetti-Rabinowitz conditions. We use cut-off function and Moser iteration to obtain the existence of nontrivial solution. Finally, as a by-product of our approaches, we get the same result for Klein-Gordon-Maxwell system.
\end{abstract}


\section{Introduction}
This paper studies the Klein-Gordon equation coupled with Born-Infeld theory with critical growth
\begin{eqnarray}\label{problem 1}
	\left\{   \begin{array}{ll}
	-\Delta u + V(x)u-(2\omega+\phi)\phi u=\lambda f(u)+|u|^{4}u,   \ & \text{in} \  \mathbb{R}^{3}, \\
	\Delta \phi + \beta\Delta_4\phi  = 4\pi(\omega+\phi) u^{2},   \ & \text{in}\  \mathbb{R}^{3},\\
\end{array}
\right.	
\end{eqnarray}
where $\omega>0$ is a constant, $\lambda>0$ is a positive parameter. Klein-Gordon equation can be used to develop the theory of electrically charged fields (see\cite{1997Geometry}) and study the interaction with an assigned electromagnetic field (see \cite{MR0489509}). The Born-Infeld (BI) electromagnetic theory \cite{born1934quantum,born1934foundations} was originally proposed as a nonlinear correction of the Maxwell theory in order to overcome the problem of infiniteness in the classical electrodynamics of point particles(see \cite{MR1936545}). Klein-Gordon equation coupled with Born-Infeld theory system has attracted many theoretic physicists. For more physical applications, please refer to reference \cite{2004Coupled,2010Solitary} and the references therein.

In the past decades, many people have studied this system through using variational methods, and have also obtained existence of nontrivial solutions under different assumptions. Let us recall some previous results which give an inspiration to the presence research.

The first result is due to d'Avenia and Pisani, in which the existence of infinitely many radially symmetric solution for the following form
\begin{eqnarray}\label{problem 2}
	\left\{   \begin{array}{ll}
	-\Delta u + [m^{2}-(\omega+\phi)^{2}] u=|u|^{p-2}u,   \ & \text{in} \  \mathbb{R}^{3}, \\
	\Delta \phi + \beta\Delta_4\phi  = 4\pi(\omega+\phi) u^{2},   \ & \text{in}\  \mathbb{R}^{3},\\
\end{array}
\right.
\end{eqnarray}
was proved when $4<p<6$ and $|\omega|<|m_{0}|$ in \cite{2002Nonlinear}. Mugnai \cite{2004Coupled} get the same result when $2<p\le 4$ and $0<\omega<\sqrt{\frac{1}{2}p-1}|m|$. Afterwards, Wang \cite{2012Solitary} use Poho\v{z}aev identity to improve literature \cite{2002Nonlinear}, \cite{2004Coupled} and obtains the solitary wave solution when the one of the following conditions is satisfied
\begin{enumerate}[(i)]
    \item $3<p<6$ and $m>\omega>0$,
    \item $2<p\le 3$ and $(p-2)(4-p)m^{2}>\omega^{2}>0$.
\end{enumerate}
Yu \cite{2010Solitary} get the existence of the least-action solitary wave in both bounded smooth case and
$\mathbb{R}^{3}$ case. Moreover, replacing $|u|^{p-2}u$ by $|u|^{p-2}u+h(x)$, Chen and Li in \cite{2013Multiple} get the existence of multiple solution if one of the following condition holds
\begin{enumerate}[(i)]
    \item $4<p<6$ and $|m|>\omega$,
    \item $2<p\le 4$ and $\sqrt{\frac{1}{2}p-1}|m|>\omega$.
\end{enumerate}
Later, Chen and Song \cite{2017The} studied the following Klein-Gordon equation with concave and convex nonlinearities coupled with Born-Infeld theory
 \begin{eqnarray}\label{problem }
	\left\{   \begin{array}{ll}
	-\Delta u + V(x)u-(2\omega+\phi)\phi u=\lambda k(x)|u|^{q-2}u+g(x)|u|^{p-2}u,   \ & \text{in} \  \mathbb{R}^{3}, \\
	\Delta \phi + \beta\Delta_4\phi  = 4\pi(\omega+\phi) u^{2},   \ & \text{in}\  \mathbb{R}^{3}.\\
\end{array}
\right.	
\end{eqnarray}
Under some appropriate assumptions on $V(x),\ \lambda,\ k(x)$ and $g(x)$, the obtain the existence of multiple nontrivial solutions when $1<q<2<p<6$. Recently, for general potential $V(x)$ and $|u|^{p-2}u$ by a continuous nonlinearity $f(x,u)$ with polynomial growth, Wen and Tang \cite{MR4001213} obtained infinitely many solutions and least energy solutions, Che and Chen \cite{MR4125947} use genus theory to obtain nontrivial solutions.

We know that a large number of predecessors have studied the problem of subcritical growth like the above papers. Furthermore, when the nonlinearity term is accompanied by critical growth, it is one of the most dramatic cases of loss of compactness. To my best knowledge, there is only one work about the Klein-Gordon-Born-Infeld system with critical growth. Teng and Zhang \cite{2011Exist}
investigated the following system
\begin{eqnarray}\label{}
	\left\{   \begin{array}{ll}
	-\Delta u + [m^{2}-(\omega+\phi)^{2}] u=|u|^{p-2}u+|u|^{2^{*}-2}u,   \ & \text{in} \  \mathbb{R}^{3}, \\
	\Delta \phi + \beta\Delta_4\phi  = 4\pi(\omega+\phi) u^{2},   \ & \text{in}\  \mathbb{R}^{3}.\\
\end{array}
\right.
\end{eqnarray}
They obtained it has at least a nontrivial solution when $4<p<6$ and $m<\omega.$

Motivated by the aforementioned works, in this paper, we will use some new tricks to generalize the above results to problem \eqref{problem 1} under the following conditions:
\begin{enumerate}[($V_{1}$)]
    \item  $V\in C^{1}(\mathbb{R}^{3},\mathbb{R})$ and there is a $V_{0}>0$ such that $V(x)\ge V_{0}$ for all $x\in\mathbb{R}^{3}.$
\end{enumerate}
\begin{enumerate}[($V_{2}$)]
    \item  $V(x)\to\infty$ as $|x|\to\infty.$
\end{enumerate}
\begin{enumerate}[($f_{1}$)]
	\item $f\in C(\mathbb{R})$, $\lim_{u\to 0}\frac{f(u)}{u}=0.$
\end{enumerate}
\begin{enumerate}[($f_{2}$)]
	\item $\lim_{|u|\to\infty}\frac{f(u)}{u}=+\infty.$
\end{enumerate}

Our first result is as follows.
\begin{theorem}\label{theorem 1.1}
    Assume that $(V_{1})-(V_{2})$ and $(f_{1})-(f_{2})$ hold. then there exists a constant $\lambda_{1}^*\ge 0$ such that, for any $\lambda\in (0,\lambda_{1}^*)$, system \eqref{problem 1} has a nontrivial solution.
\end{theorem}
\begin{remark}
We all know that if the nonlinear term is $|u|^{4}u$, we can use the Poho\v{z}aev identity and the classical variational method to know whether the system has no nontrivial solutions. Therefore, when studying the nonlinear term is a critical growth case, it is usually to add a high energy lower-order perturbation term like \cite{MR2754298,2011Exist,2011Solitary}. By comparison with the above papers, the result of this paper is that $\lambda$ is small enough, that is, the lower-order perturbation is a lower energy perturbation.
\end{remark}
\begin{remark}
    \begin{enumerate}[(i)]
    \item The condition $(V_{2})$ was first introduced by P.H.Rabinowitz in \cite{MR1162728} to overcome the lack of compactness.
    \item It is worth noting that in this paper we did not require any growth conditions and the Ambrosetti-Rabinowitz condition, and the function $f$ can also be sign-changing.
    \item  There are many functions that can satisfy the condition $(f_{1})-(f_{2}),$ the most typical example is $f(t)=|t|^{p-2}t,$ $p>6$. Moreover, our result is valid for general supercritical nonlinearity.
\end{enumerate}
\end{remark}
We emphasize that our result requires no growth conditions. To prove the existence of nontrivial solution, we adapt a similar argument as in \cite{MR3950621,MR4143620}. Here we briefly explain the process. Firstly, we make a suitable cut-off function to replace $f(u)$ in problem \eqref{problem 1}, so we can get a new system. Secondly, we prove the new system have nontrivial solution. Finally we use the Moser iteration to obtain the existence of nontrivial solution to original Klein-Gordon equation coupled with Born-Infeld theory.

In the second part of this paper, It is worthy of our special attention that when $\beta=0$, a small modification to problem \eqref{problem 1} will become a Klein-Gordon-Maxwell system with critical growth, namely:
\begin{eqnarray}\label{KGM}
	\left\{   \begin{array}{ll}
	-\Delta u + V(x)u-(2\omega+\phi)\phi u=\lambda f(u)+|u|^{4}u,   \ & \text{in} \  \mathbb{R}^{3}, \\
	\Delta \phi= (\omega+\phi) u^{2},   \ & \text{in}\  \mathbb{R}^{3},\\
\end{array}
\right.	
\end{eqnarray}
which has been extensively studied by many authors. A pioneer work is due to Cassani \cite{MR2085333} considered the following critical Klein-Gordon-Maxwell system:
\begin{eqnarray*}
	\left\{   \begin{array}{ll}
	-\Delta u + [m^{2}-(\omega+\phi^{2})]u=\lambda |u|^{p-2}u +|u|^{4}u,   \ & \text{in} \  \mathbb{R}^{N}, \\
	\Delta \phi= (\omega+\phi) u^{2},   \ & \text{in}\  \mathbb{R}^{N}.\\
\end{array}
\right.	
\end{eqnarray*}
where $\lambda>0,\ 2<p<6$ and $0<\omega<m$. When $N=3$ he obtained the existence of a radially symmetric solution for any $\lambda>0$ if $p\in(4,6)$ and for $\lambda$ is sufficiently large if $p=4$. Afterwards C.Carriao, L.Cunha and H.Miyagaki \cite{MR2754298} complement the result of \cite{MR2085333} and also extend it in higher dimensions. They obtained the same result provided one of those conditions satisfies
\begin{enumerate}[(i)]
    \item $N=4$ and $N\ge 6$ for $2<p<2^*$ and $|m|>\omega$ if $\lambda>0$;
    \item $N=5$ and either $2<p<\frac{8}{3}$ if $\lambda>0$ or $\frac{8}{3}\le p<2^*$ if $\lambda$ is sufficiently large;
    \item $N=3$ and either $4<p<2^*$ if $\lambda>0$ or $2<p\le 4$ if $\lambda$ is sufficiently large.
\end{enumerate}
Later,When $N=3$ Wang\cite{2011Solitary} improved the result of \cite{MR2754298,MR2085333} to the case when one of the following holds:
\begin{enumerate}[(i)]
    \item $4<p<6$, $0<\omega<m$ and $\lambda>0$;
    \item $3<p\le 4$, $0<\omega<m$ and $\lambda$ is sufficiently large;
    \item $2<p\le 3$, $0<\omega<\sqrt{(p-2)(4-p)}m$ and $\lambda$ is sufficiently large.
\end{enumerate}
In recent paper \cite{MR3895092}, Chen uses some analytical skills and variational method that is different from \cite{2011Solitary} to get the same result of it.
The authors of \cite{MR2754298} have also studied that for problem \eqref{KGM}, $N=3$, $V(x)$ is a periodic function and $f(u)=|u|^{p-2}u$ in \cite{MR2914593}. They use the minimization of the corresponding Euler-Lagrange functional on the Nehari manifold and the Br$\acute{e}$zis and Nirenberg technique to get a positive ground state solution for each $\lambda>0$ if $p\in (4,6)$ and for $\lambda$ sufficiently large if $p\in (2,4]$. Moreover when the potential well is steep, namely
\begin{eqnarray}\label{steep}
	\left\{   \begin{array}{ll}
	-\Delta u + \mu V(x)u-(2\omega+\phi)\phi u=\lambda f(x,u)+|u|^{4}u,   \ & \text{in} \  \mathbb{R}^{3}, \\
	\Delta \phi= (\omega+\phi) u^{2},   \ & \text{in}\  \mathbb{R}^{3}.\\
\end{array}
\right.	
\end{eqnarray}
where $\mu$, $\lambda$ are positive parameters and $\omega>0$, there exist $\hat{\mu_{0}},\ \hat{\lambda_{0}}>0$ such that for $\mu>\hat{\mu_{0}}$ and $\lambda>\hat{\lambda_{0}}$ problem \eqref{steep} admits a nontrivial solution has been proved by Zhang in \cite{MR3671218}. At the same time, he also obtained a nontrivial solution when the potential well may be not steep. Instead of the expression "$\lambda$ sufficiently large" in the above existing works, Tang, Wen and Chen\cite{MR4057065} give a certain range $\lambda\ge\lambda_{0}$ which admites a ground state solution when $V$ is positive and periodic.

Similarly to the method of Theorem \ref{problem 1}, we can also get a nontrivial solution. Compared with the hypothesis of subcritical perturbation in the above article, in this paper the perturbation term $f(u)$ can be not only a subcritical perturbation but also a supercritical perturbation. What's more the restriction on $\lambda$ is no longer sufficiently large or greater than a certain number, we can only require $\lambda\in (0,\lambda_{2}^*)$ where $\lambda_{2}^*\ge0$. Our second result is as follows.
\begin{theorem}\label{corollary 1.2}
    Assume that $(V_{1})-(V_{2})$ and $(f_{1})-(f_{2})$ hold. then there exists a constant $\lambda_{2}^*\ge 0$ such that, for any $\lambda\in (0,\lambda_{2}^*)$, system \eqref{KGM} has a nontrivial solution.
\end{theorem}
\begin{remark}
We underline that the existence of nontrivial solution for problem \eqref{KGM} it has been proved by above papers with a different approach in this paper. However, it is interesting that we do not need $\lambda$ is sufficiently large or greater than a certain number.
\end{remark}

This paper is organized as follows. In Section $2$, we give some preliminary lemmas. In section $3$, we prove Theorems \ref{theorem 1.1}. In section $4$, we prove Theorems \ref{corollary 1.2}.
\section{Preliminaries}
In this section we explain the notations and some auxiliary lemmas which are useful later.

$H^{1}(\mathbb{R}^{3})$ denotes the usual Sobolev space equipped with the standard norm.

$L^{\ell}(\mathbb{R}^{3})$, $\ell\in[1,+\infty)$ denotes the Lebesgue space with the norm
$|u|_{\ell}=\left(\int_{\mathbb{R}^{3}}|u|^{\ell}dx\right)^{\frac{1}{\ell}}$.

Under $(V_{1})$ and $(V_{2})$, we define the Hilbert space
$$ E=\left\{u\in H^{1}(\mathbb{R}^{3}):\int_{\mathbb{R}^{3}}V(x)u^{2}dx<\infty\right\},$$
with respect the norm
$$\left\|u\right\|=\left (\int_{\mathbb{R}^{3}}\left(\left |\nabla u \right |^{2}+V(x)u^{2}\right)dx\right )^{\frac{1}{2}}.$$
Then, the embedding $E\hookrightarrow H^{1}(\mathbb{R}^{3})$ is continuous. The embedding from $E$ into $L^{q}(\mathbb{R}^{3})$ is compact for $q\in [2,6)$ and its detailed proof process can be seen in Lemma $3.4$ in \cite{MR2232879}.

Denote by $D(\mathbb{R}^{3})$ the completion of $C_{0}^{\infty}(\mathbb{R}^{3})$ with respect to the norm
$$\left\|\phi\right\|_{D\left (\mathbb{R}^{3}\right )}=\left (\int_{\mathbb{R}^{3}}\left |\nabla \phi \right |^{2} dx\right ) ^{\frac{1}{2}}+ \left (\int_{\mathbb{R}^{3}}\left |\nabla \phi \right |^{4} dx\right ) ^{\frac{1}{4}}.$$ It is easy to know that $D\left (\mathbb{R}^{3}\right )$ is continuously embedded in $D^{1,2}\left (\mathbb{R}^{3}\right )$, where $D^{1,2}\left (\mathbb{R}^{3}\right )$ is the completion of $C_{0}^{\infty}(\mathbb{R}^{3})$ with respect to the norm
$\left\|\phi\right\|_{D^{1,2}\left (\mathbb{R}^{3}\right )}=\left (\int_{\mathbb{R}^{3}}\left |\nabla \phi \right |^{2} dx\right ) ^{\frac{1}{2}}.$ Moreover, $D^{1,2}\left (\mathbb{R}^{3}\right )$ is continuously embedded in $L^{6}\left(\mathbb{R}^{3}\right)$ by Sobolev inequality and $D\left (\mathbb{R}^{3}\right )$ is continuously embedded in $L^{\infty}(\mathbb{R}^{3})$.

$C_{1},\ C_{2},\ \cdot\cdot\cdot$ denote positive constant possibly different in different places.

Indeed, solutions of \eqref{problem 1} are critical point of functional $G_{\lambda}:E(\mathbb{R}^{3})\times D(\mathbb{R}^{3}) \to \mathbb{R}$, defined by
\begin{eqnarray}\label{e1}
    \begin{split}
		G_{\lambda}(u,\phi)=&\frac{1}{2}\int_{\mathbb{R}^{3}}\left(\left | \nabla u\right |^{2}+V(x)u^{2}-(2\omega+\phi)\phi u^{2}\right)dx-\dfrac{1}{8\pi}\int_{\mathbb{R}^{3}}\left|\nabla \phi\right|^{2}dx \\
		&-\dfrac{\beta}{16\pi}\int_{\mathbb{R}^{3}}\left |\nabla\phi\right |^{4}dx-\int_{\mathbb{R}^{3}}\left(\lambda F(u)+\frac{1}{6}|u|^{6}\right)dx.
	\end{split}
\end{eqnarray}
Due to the strong indefiniteness of functional \eqref{e1}, we use the reduction method which can reduces the study of $G_{\lambda}(u,\phi)$ to study a new functional $I_{\lambda}(u)$ as in \cite{MR1714281}.

We state some properties of the second equation of problem \eqref{problem 1}.
\begin{lemma}\label{L 1}
	For any $u\in H^{1}\left(\mathbb{R}^{3}\right)$, we have:
	\begin{enumerate}[(i)]
		\item there exists a unique $\phi_{u}\in D\left(\mathbb{R}^{3}\right)$ which solves the second equation of problem \eqref{problem 1}.
		\item in the set $\left\{X: u(x)\ne 0\right\}$, we have $-\omega \le\phi_{u}\le 0.$
        \item $\left\|\phi_{u}\right\|_{D}\le C\left\|u\right\|^{2}$ and $\int_{\mathbb{R}^{3}}|\phi_{u}|u^{2}dx\le C\left\|u\right\|_{\frac{12}{5}}^{4}$.
	\end{enumerate}
\end{lemma}
\begin{proof}
(i) is proved in Lemma $3$ of \cite{2002Nonlinear} and (ii) can be found in Lemma $2.3$ of \cite{2004Coupled}.
\begin{align}\label{4586}
    \begin{split}
   \int_{\mathbb{R}^{3}}\left| \nabla \phi_{u}\right| ^{2}dx+\beta\int_{\mathbb{R}^{3}}\left| \nabla \phi_{u}\right| ^{4}dx&=-\int_{\mathbb{R}^{3}}4\pi\omega\phi_{u} u^{2}dx-\int_{\mathbb{R}^{3}}4\pi\phi_{u}^{2} u^{2}dx
    \\&\le 4\pi\omega\int_{\mathbb{R}^{3}}|\phi_{u}| u^{2}dx
    \\&\le 4\pi\omega \left\|\phi_{u}\right\|_{D}\left\|u\right\|_{\frac{12}{5}}^{2}.
    \end{split}
\end{align}
we can get $\left\|\phi_{u}\right\|_{D}\le C\left\|u\right\|^{2}$ and $\int_{\mathbb{R}^{3}}|\phi_{u}|u^{2}dx\le C\left\|u\right\|_{\frac{12}{5}}^{4}$.
\end{proof}
From the second equation in \eqref{problem 1} and Lemma \ref{L 1} , we get
\begin{eqnarray}\label{e 4}
\frac{1}{4\pi}\int_{\mathbb{R}^{3}}\left| \nabla \phi_{u}\right| ^{2}dx+\frac{\beta}{4\pi}\int_{\mathbb{R}^{3}}\left| \nabla \phi_{u}\right| ^{4}dx=-\int_{\mathbb{R}^{3}}\left(\omega\phi_{u}+\phi_{u}^{2}\right)u^{2}dx.
\end{eqnarray}
Consider the functional $I_{\lambda}(u):E\to \mathbb{R}$ defined by
$I_{\lambda}(u)=G_{\lambda}\left (u,\phi_{u}\right )$ and combine \eqref{e 4}, we obtain
\begin{align}\label{e 3}
	\begin{split}
		I_{\lambda}(u)&=\frac{1}{2}\int_{\mathbb{R}^{3}}\left(\left | \nabla u\right |^{2}+V(x)u^{2}-(2\omega+\phi_{u})\phi_{u}u^{2}\right)dx-\frac{1}{8\pi}\int_{\mathbb{R}^{3}}|\nabla \phi_{u}|^{2}dx-\frac{\beta}{16\pi}\int_{\mathbb{R}^{3}}|\nabla \phi_{u}|^{4}dx \\
		&\ \ \ -\int_{\mathbb{R}^{3}}\left(\lambda F(u)+\frac{1}{6}|u|^{6}\right)dx
\\&=\frac{1}{2}\int_{\mathbb{R}^{3}}\left(\left| \nabla u\right| ^{2}+V(x)u^{2}\right)dx-\frac{3}{4}\int_{\mathbb{R}^{3}}\omega\phi_{u}u^{2}dx-\frac{1}{4}
\int_{\mathbb{R}^{3}}\phi_{u}^{2}u^{2}dx-\frac{1}{16\pi}\int_{\mathbb{R}^{3}}|\nabla\phi_{u}|^{2}dx\\
&\ \ \ -\int_{\mathbb{R}^{3}}\left(\lambda F(u)+\frac{1}{6}|u|^{6}\right)dx
\\&=\frac{1}{2}\int_{\mathbb{R}^{3}}\left(\left| \nabla u\right| ^{2}+V(x)u^{2}\right)dx-\frac{1}{2}\int_{\mathbb{R}^{3}} \omega\phi_{u}u^{2}dx+\frac{\beta}{16\pi}\int_{\mathbb{R}^{3}}\left| \nabla \phi_{u}\right| ^{4}dx\\&
\ \ \ -\int_{\mathbb{R}^{3}}\left(\lambda F(u)+\frac{1}{6}|u|^{6}\right)dx.
	\end{split}
\end{align}
Of course $I_{\lambda}(u)\in C^{1}\left(E,\mathbb{R}\right)$ and for any $u,v\in E$, we have
\begin{eqnarray}\label{qruation 51}
	\langle I_{\lambda}'(u),v \rangle=\int_{\mathbb{R}^{3}} \left\{\nabla u \cdot\nabla v+V(x)uv-(2\omega+\phi_{u})\phi_{u} uv-\lambda f(u)v-|u|^{4}uv\right\}dx.
\end{eqnarray}
\begin{lemma}(\cite{2004Coupled})
    The following statements are equivalent:
    \begin{enumerate}[(i)]
        \item $(u,\phi)\in E(\mathbb{R}^{3})\times D(\mathbb{R}^{3})$ is a critical point of $G_{\lambda}$, i.e. $(u,\phi)$ is a solution of problem \eqref{problem 1};
        \item $u$ is a critical point of $I_{\lambda}$ and $\phi=\phi_{u}$.
    \end{enumerate}
\end{lemma}

From $(f_{2})$ we deuce exist $T>0$ large enough such that $f(T)>0$. Let
\begin{eqnarray}\label{a1}
	h_{T}(t)=\left\{   \begin{array}{ll}
	f(t),\ &0< t\le T, \\
    C_{T}t^{p-1},\ & t>T,\\
    0,\ & t\le 0,
\end{array}
\right.	
\end{eqnarray}
where $f$ satisfies $(f_{1})\ (f_{2})$ and $C_{T}=\frac{f(T)}{T^{p-1}}\ (4<p<6)$. $h_{T}$ is a continuous function and satisfies the following properties:
\begin{enumerate}[($h_{1}$)]
	\item $\lim_{t\to 0^{+}}\frac{h_{T}(t)}{t}=0.$
\end{enumerate}
\begin{enumerate}[($h_{2}$)]
	\item $\lim_{t\to +\infty}\frac{H_{T}(t)}{t^{4}}=+\infty,$ where $H_{T}(t)=\int_{0}^{t}h_{T}(s)ds.$
\end{enumerate}
\begin{enumerate}[($h_{3}$)]
	\item $|h_{T}(t)|\le C_{T}^{*}|t|+C_{T}|t|^{p-1},$ where $C_{T}^{*}=\max_{t\in [0,T]}\frac{|f(t)|}{t}$.
\end{enumerate}
\begin{enumerate}[($h_{4}$)]
	\item There exists $\mu=\mu(T)>0$ such that $th_{T}(t)-4H_{T}(t)\ge -\mu t^{2}$ for all $t\geq0.$
\end{enumerate}
Next, we will use the cut-off functional $h$ to replace $f$ in problem \eqref{problem 1} and combine Lemma \ref{L 1}. We can get a new problem, namely
\begin{eqnarray}\label{new problem}
	\left\{   \begin{array}{ll}
	-\Delta u + V(x)u-(2\omega+\phi_{u})\phi_{u} u=\lambda h_{T}(u)+|u|^{4}u,   \ & x\in \  \mathbb{R}^{3}, \\
	u(x)>0,   \ & u\in E.\\
\end{array}
\right.	
\end{eqnarray}
we will study critical points for the functional
\begin{align}\label{277}
	\begin{split}
		J_{\lambda,T}(u)&=\frac{1}{2}\int_{\mathbb{R}^{3}}\left(\left | \nabla u\right |^{2}+V(x)u^{2}-(2\omega+\phi_{u})\phi_{u}u^{2}\right)dx-\frac{1}{8\pi}\int_{\mathbb{R}^{3}}|\nabla \phi_{u}|^{2}dx-\frac{\beta}{16\pi}\int_{\mathbb{R}^{3}}|\nabla \phi_{u}|^{4}dx \\
		&\ \ \ -\int_{\mathbb{R}^{3}}\left(\lambda H_{T}(u)+\frac{1}{6}|u|^{6}\right)dx
\\&=\frac{1}{2}\int_{\mathbb{R}^{3}}\left(\left| \nabla u\right| ^{2}+V(x)u^{2}\right)dx-\frac{3}{4}\int_{\mathbb{R}^{3}}\omega\phi_{u}u^{2}dx-\frac{1}{4}
\int_{\mathbb{R}^{3}}\phi_{u}^{2}u^{2}dx-\frac{1}{16\pi}\int_{\mathbb{R}^{3}}|\nabla\phi_{u}|^{2}dx\\
&\ \ \ -\int_{\mathbb{R}^{3}}\left(\lambda H_{T}(u)+\frac{1}{6}|u|^{6}\right)dx
\\&=\frac{1}{2}\int_{\mathbb{R}^{3}}\left(\left| \nabla u\right| ^{2}+V(x)u^{2}\right)dx-\frac{1}{2}\int_{\mathbb{R}^{3}} \omega\phi_{u}u^{2}dx+\frac{\beta}{16\pi}\int_{\mathbb{R}^{3}}\left| \nabla \phi_{u}\right| ^{4}dx\\&
\ \ \ -\int_{\mathbb{R}^{3}}\left(\lambda H_{T}(u)+\frac{1}{6}|u|^{6}\right)dx
	\end{split}
\end{align}
as solutions to \eqref{new problem}. Through direct calculation ,we know that the function $h_{T}$ is continuous, so we have $J_{\lambda,T}\in C^{1}(E,\mathbb{R})$ and for any $u,v\in E$,
\begin{eqnarray}\label{qruation 5}
	\langle J_{\lambda,T}'(u),v \rangle=\int_{\mathbb{R}^{3}} \left\{\nabla u \cdot\nabla v+V(x)uv-(2\omega+\phi_{u})\phi_{u}uv-\lambda h_{T}(u)v-|u|^{4}uv\right\}dx.
\end{eqnarray}
The next lemma shows that functional $J_{\lambda,T}(u)$ satisfies the mountain pass geometry.
\begin{lemma}\label{geometric}
	The functional $J_{\lambda , T}(u)$ satisfies the following conditions:
	\begin{enumerate}[(i)]
		\item there exists $\alpha,\rho >0$ such that $J_{\lambda ,T}(u)\ge\alpha$ when $\left\|u\right\|=\rho$;
		\item there exists $e\in E$ such that $\left\|e\right\|>\rho$ and $J_{\lambda,T}(e)<0.$
	\end{enumerate}
\end{lemma}
\begin{proof}
From $(h_{1})$ and $(h_{3})$, there exists a  $\varepsilon>0$ small such that
$$|h_{T}(t)|\le\varepsilon|t|+C_{\varepsilon}|t|^{5},$$
and
\begin{align}\label{295}
|H_{T}(t)|\le\frac{\varepsilon}{2}|t|^{2}+\frac{C_{\varepsilon}}{6}|t|^{6}.
\end{align}
Then from Lemma \ref{L 1}, \eqref{277}, \eqref{295} and Sobolev embedding theorem, for every $u\in E\setminus\left\{0\right\}$ we can deduce
\begin{align}\label{}
	\begin{split}
		J_{\lambda,T}(u)&=\frac{1}{2}\int_{\mathbb{R}^{3}}\left(\left| \nabla u\right| ^{2}+V(x)u^{2}\right)dx-\frac{1}{2}\int_{\mathbb{R}^{3}} \omega\phi_{u}u^{2}dx+\frac{\beta}{16\pi}\int_{\mathbb{R}^{3}}\left| \nabla \phi_{u}\right| ^{4}dx\\&
\ \ \ -\int_{\mathbb{R}^{3}}\left(\lambda H_{T}(u)+\frac{1}{6}|u|^{6}\right)dx\\&
\ \ \ \ge\frac{1}{2}\int_{\mathbb{R}^{3}}\left(\left| \nabla u\right| ^{2}+V(x)u^{2}\right)dx-\int_{\mathbb{R}^{3}}\left(\lambda H_{T}(u)+\frac{1}{6}|u|^{6}\right)dx\\&
\ \ \ \ge\frac{1}{2}\left\|u\right\|^{2}-\int_{\mathbb{R}^{3}}\left(\frac{\lambda\varepsilon}{2}
|u|^{2}+\frac{\lambda C_{\varepsilon}}{6}|u|^{6}+\frac{1}{6}|u|^{6}\right)dx\\&\ \ \
\ge\frac{1}{2}\left\|u\right\|^{2}-C\varepsilon\left\|u\right\|^{2}-C\left\|u\right\|^{6}.
	\end{split}
\end{align}
Since $\varepsilon$ is arbitrarily small, there exists $\rho>0$ and $\alpha>0$ such that $J_{\lambda.T}(u)\ge\alpha>0$ for $\left\|u\right\|=\rho.$ Hence, $J_{\lambda.T}(u)$ satisfied $(i)$.

From $(h_{1})$ $(h_{2})$ and $(h_{3})$ we get for any $M>0$ there exists a positive constant $C_{M}>0$ such that
\begin{align}\label{213}
H_{T}(u)\ge Mt^{4}-C_{M}t^{2},\ \ \ \forall t>0.
\end{align}
So, fix $u\in E\setminus\left\{0\right\}$ and $t>0$, From \eqref{277} and \eqref{213} we obtain
\begin{align}\label{}
	\begin{split}
		J_{\lambda,T}(tu)&=\frac{t^{2}}{2}\int_{\mathbb{R}^{3}}\left(\left | \nabla u\right |^{2}+V(x)u^{2}\right)dx-t^{2}\int_{\mathbb{R}^{3}}\omega\phi_{tu}u^{2}dx-
\frac{t^{2}}{2}\int_{\mathbb{R}^{3}}\phi_{tu}^{2}u^{2}-
\frac{1}{8\pi}\int_{\mathbb{R}^{3}}|\nabla \phi_{tu}|^{2}dx\\
&\ \ \ -\frac{\beta}{16\pi}\int_{\mathbb{R}^{3}}|\nabla \phi_{tu}|^{4}dx-\int_{\mathbb{R}^{3}}\left(\lambda H_{T}(tu)+\frac{t^{6}}{6}|u|^{6}\right)dx\\
& \le\frac{t^{2}}{2}\int_{\mathbb{R}^{3}}\left(\left | \nabla u\right |^{2}+V(x)u^{2}\right)dx-t^{2}\int_{\mathbb{R}^{3}}\omega\phi_{tu}u^{2}dx+
\lambda C_{M}t^{2}\int_{\mathbb{R}^{3}}|u|^{2}dx\\
&\ \ \ -\lambda Mt^{4}\int_{\mathbb{R}^{3}}|u|^{4}dx-\frac{t^{6}}{6}\int_{\mathbb{R}^{3}}|u|^{6}dx.
	\end{split}
\end{align}	
we can easy to know that $J_{\lambda,T}(tu)\to -\infty$ for $t\to +\infty.$ The step (ii) is proved by taking $e=t_{0}u$ with $t_{0}>0$ large enough.
\end{proof}
From Lemma \ref{geometric}, we can easily get a $PS$ sequence, namely there exists a sequence $\left\{u_{n}\right\}\subset E$ satisfying
\begin{align}\label{Psc}
J_{\lambda,T}(u_{n})\to c_{\lambda,T},\ \ \
J'_{\lambda,T}(u_{n})\to 0,
\end{align}
where
$$c_{\lambda,T}:=\inf_{\gamma\in\Gamma}\max_{0\le t\le 1}J_{\lambda,T}(\gamma(t)),$$
$$\Gamma:=\left\{\gamma\in C\left([0,1],H^{1}(\mathbb{R}^{3})\right):\gamma(0)=0,\ \gamma(1)=e\right\}.$$
\begin{lemma}\label{youjie}
	The sequence $\left\{u_{n}\right\}$ defined by \eqref{Psc} is bounded in $E$.
\end{lemma}
\begin{proof}
From \eqref{277}, \eqref{qruation 5}, \eqref{Psc} and $(h_{4})$, we obtain that
\begin{align}\label{AS}
	\begin{split}
		&c_{\lambda,T}+o_{n}(1)\left\|u_{n}\right\|\\
        &\ \ \ \ge J_{\lambda,T}(u_{n})-\frac{1}{4}J'_{\lambda,T}(u_{n})\\
        &\ \ \ =\frac{1}{4}\int_{\mathbb{R}^{3}}\left(\left| \nabla u_{n}\right| ^{2}+V(x)u_{n}^{2}\right)dx+
        \frac{\beta}{16\pi}\int_{\mathbb{R}^{3}}|\nabla\phi_{u_{n}}|^{4}dx
        +\frac{1}{4}\int_{\mathbb{R}^{3}}\phi_{u_{n}}^{2}u_{n}^{2}dx\\&\ \ \ \ \ \ +\frac{\lambda}{4}\int_{\mathbb{R}^{3}}\left(h_{T}(u_{n})u_{n}-4H_{T} (u_{n})\right)dx+\frac{1}{12}\int_{\mathbb{R}^{3}}|u_{n}|^{6}dx\\
        &\ \ \ \ge\frac{1}{4}\int_{\mathbb{R}^{3}}\left(\left| \nabla u_{n}\right| ^{2}+V(x)u_{n}^{2}\right)dx+
        \frac{\lambda}{4}\int_{\mathbb{R}^{3}}\left(h_{T}(u_{n})u_{n}-4H_{T} (u_{n})\right)dx\\
        &\ \ \ \ge\frac{1}{4}\left\|u_{n}\right\|^{2}-\frac{\lambda\mu}{4}\int_{\mathbb{R}^{3}}|u_{n}^{+}|^{2}dx,	\end{split}
\end{align}
where $u_{n}(x)=u_{n}^{+}(x)+u_{n}^{-}(x)$, $u_{n}^{+}(x)=\max \left\{u_{n}(x),0\right\}$,$u_{n}^{-}(x)=\min\left\{u_{n}(x),0\right\}$.
We use the contradiction method that assume $\left\|u_{n}\right\|\to +\infty$ as $n\to\infty$. Let $v_{n}=\frac{u_{n}}{\left\|u_{n}\right\|}$, $n\ge 1$.

Due to $E\hookrightarrow L^{q}(\mathbb{R}^{3})$, $q\in[2,6)$ is compact, it is easy to assume $v_{n}\rightharpoonup v$ in $E$, then
\begin{align}
	\begin{split}
&v_{n}\to v,\ \ \ \ in\ L^{q}(\mathbb{R}^{3}),\ 2\le q< 6,\\&v_{n}\to v,\ \ \ \ a.e.\ in\ \mathbb{R}^{3}.
    \end{split}
\end{align}
What's more, we deduce
\begin{align}\label{21777}
	\begin{split}
&v^{+}_{n}\rightharpoonup v^{+},\ \ \ \ in\ E,\\
&v^{+}_{n}\to v^{+},\ \ \ \ in\ L^{q}(\mathbb{R}^{3}),\ 2\le q< 6,\\
&v^{+}_{n}\to v^{+},\ \ \ \ a.e.\ in\ \mathbb{R}^{3}.
    \end{split}
\end{align}
Divide both sided of \eqref{AS} by $\left\|u_{n}\right\|^{2}$, we obtain
\begin{align}\label{217}
    \begin{split}
    o_{n}(1)&\ge\frac{1}{4}-\frac{\lambda\mu}{4}\int_{\mathbb{R}^{3}}|v^{+}_{n}|^{2}dx
    \\&=\frac{1}{4}-\frac{\lambda\mu}{4}\int_{\mathbb{R}^{3}}|v^{+}|^{2}dx+o(1).
    \end{split}
\end{align}
we can deduce $v^{+}\neq 0$. Due to $u^{+}_{n}=v^{+}_{n}\left\|u_{n}\right\|\to +\infty,$ \eqref{qruation 5}, \eqref{Psc} and Lemma \ref{L 1}, we get
\begin{align}\label{218}
    \begin{split}
    &\frac{J'_{\lambda,T}(u_{n})u_{n}}{\left\|u_{n}\right\|^{4}}\\
    &\ \ =\frac{\left\|u_{n}\right\|^{2}}{\left\|u_{n}\right\|^{4}}
    -\frac{\int_{\mathbb{R}^{3}}2\omega\phi_{u_{n}}u_{n}^{2}dx}{\left\|u_{n}\right\|^{4}}
    -\frac{\int_{\mathbb{R}^{3}}\phi_{u_{n}}^{2}u_{n}^{2}dx}{\left\|u_{n}\right\|^{4}}
    -\frac{\int_{\mathbb{R}^{3}}\lambda h_{T}(u_{n})u_{n}dx}{\left\|u_{n}\right\|^{4}}
    -\frac{\int_{\mathbb{R}^{3}}|u_{n}|^{6}dx}{\left\|u_{n}\right\|^{4}}\\
    &\ \ \le o_{n}(1)
    +\frac{\int_{\mathbb{R}^{3}}2\omega|\phi_{u_{n}}|u_{n}^{2}dx}{\left\|u_{n}\right\|^{4}}
    -\int_{\mathbb{R}^{3}}\frac{\lambda h_{T}(u^{+}_{n})u_{n}^{+}(v_{n}^{+})^{4}}{(u_{n}^{+})^{4}}dx.
    \end{split}
\end{align}
From Lemma \ref{L 1} (iii), \eqref{21777} and $(h_{2})$, we know that $\frac{\int_{\mathbb{R}^{3}}2\omega|\phi_{u_{n}}|u_{n}^{2}dx}{\left\|u_{n}\right\|^{4}}\to 2\omega$ and $\int_{\mathbb{R}^{3}}\frac{\lambda h_{T}(u^{+}_{n})u_{n}^{+}(v_{n}^{+})^{4}}{(u_{n}^{+})^{4}}dx\to +\infty$. Taking the limit of \eqref{218} we get $0\le -\infty$ which have a contradiction. Therefore, ${u_{n}}$ is bounded in $E$.
\end{proof}
\begin{lemma}\label{FG}
If $u_{n}$ is bounded in $E$ then, up to subsequence, $\phi_{u_{n}}\to\phi_{u}$ in $D$.
\end{lemma}
\begin{proof}
Due to $u_{n}$ is bounded in $E$, we know
\begin{align}\label{A}
    \begin{split}
    &u_{n}\rightharpoonup u\ \ \ weakly\ in\ E\\
    &u_{n}\to u\ \ \ in\ L^{q}(\mathbb{R}^{3}),\ \ \ 2\le q< 6.
    \end{split}
\end{align}
From \eqref{4586}, we can easily know $\left\{\phi_{u_{n}}\right\}$ is bounded in $D(\mathbb{R}^{3})$. So, there exists $\phi_{0}\in D$ such that $\phi_{u_{n}}\rightharpoonup\phi_{0}$ in $D$, as a consequence,
\begin{align}\label{B}
    \begin{split}
   &\phi_{u_{n}}\rightharpoonup\phi_{0}\ \ \ weakly\ in\ L^{6}(\mathbb{R}^{3}),\\
   &\phi_{u_{n}}\to\phi_{0}\ \ \ in\ L_{loc}^{q}(\mathbb{R}^{3}),\ \ 1\le q<6.
    \end{split}
\end{align}
Next we will show $\phi_{u}=\phi_{0}$.
By Lemma \eqref{L 1}, it suffices to show that
$$\Delta\phi_{0}+\beta\Delta_{4}\phi_{0}=4\pi(\omega+\phi_{0})u^{2}.$$
Let $\varphi\in C_{0}^{\infty}(\mathbb{R}^{3})$ be a test function. Since $\Delta\phi_{u_{n}}+\beta\Delta_{4}\phi_{u_{n}}=4\pi(\omega+\phi_{u_{n}})u^{2},$ we get
$$-\int_{\mathbb{R}^{3}}\langle \nabla\phi_{u_{n}},\nabla\varphi \rangle dx-\beta
\int_{\mathbb{R}^{3}}\langle |\nabla\phi_{u_{n}}|^{2}\nabla\phi_{u_{n}},\nabla\varphi \rangle dx=\int_{\mathbb{R}^{3}}4\pi\omega u_{n}^{2}\varphi dx+\int_{\mathbb{R}^{3}}4\pi\phi_{u_{n}} u_{n}^{2}\varphi dx.$$
From \eqref{A} and \eqref{B} and the boundedness of $\left\{\phi_{u_{n}}\right\}$ in $D$, the following formulas are all true, namely
$$\int_{\mathbb{R}^{3}}\langle \nabla\phi_{u_{n}},\nabla\varphi \rangle dx\stackrel{n\to\infty}{\longrightarrow}\int_{\mathbb{R}^{3}}\langle \nabla\phi_{0},\nabla\varphi \rangle dx,$$
$$\int_{\mathbb{R}^{3}}\langle |\nabla\phi_{u_{n}}|^{2}\nabla\phi_{u_{n}},\nabla\varphi \rangle dx\stackrel{n\to\infty}{\longrightarrow}\int_{\mathbb{R}^{3}}\langle |\nabla\phi_{0}|^{2}\nabla\phi_{0},\nabla\varphi \rangle dx,$$
$$\int_{\mathbb{R}^{3}} u_{n}^{2}\varphi dx\stackrel{n\to\infty}{\longrightarrow}\int_{\mathbb{R}^{3}}u^{2}\varphi dx,$$
$$\int_{\mathbb{R}^{3}}4\pi\phi_{u_{n}} u_{n}^{2}\varphi dx\stackrel{n\to\infty}{\longrightarrow}\int_{\mathbb{R}^{3}}4\pi\phi_{0}u^{2}\varphi dx,$$
proving that $\phi_{u}=\phi_{0}.$

Since $\phi_{u_{n}}$ and $\phi_{u}$ satisfies the second equation in problem \eqref{problem 1}, let us take the difference between them, we get
\begin{align}
	\begin{split}
    &\int_{\mathbb{R}^{3}}\left\{\nabla(\phi_{u_{n}}-\phi_{u})\nabla v+\beta(|\nabla\phi_{u_{n}}|^{2}\nabla\phi_{u_{n}}
    -|\nabla\phi_{u}|^{2}\nabla\phi_{u})\nabla v\right\}dx\\&
    \ \ =-4\pi\int_{\mathbb{R}^{3}}\left\{\omega(u_{n}^{2}
    -u^{2})v+(\phi_{u_{n}}u_{n}^{2}-\phi_{u}u^{2})v\right\}dx
    \end{split}
\end{align}
for any $v\in D$. Let $v=\phi_{u_{n}}-\phi_{u}$ and using the inequality
$$(|x|^{p-2}x-|y|^{p-2}y)(x-y)\ge c_{p}|x-y|^{p},\ \ \ for\ any\ x,y\in\mathbb{R}^{N},\ p\ge 2$$
the following hold
\begin{align}
	\begin{split}
    &C(\left\|\nabla\phi_{u_{n}}-\nabla\phi_{u}\right\|_{2}^{2}+
    \left\|\nabla\phi_{u_{n}}-\nabla\phi_{u}\right\|_{4}^{4})\\&
    \le 4\pi\int_{\mathbb{R}^{3}}\left(\omega|u_{n}^{2}-u^{2}|
    |\phi_{u_{n}}-\phi_{u}|+|\phi_{u_{n}}||\phi_{u_{n}}-\phi_{u}|u_{n}^{2}
    +|\phi_{u}||\phi_{u_{n}}-\phi_{u}|u^{2}\right)dx.
    \end{split}
\end{align}
By the H\"{o}lder inequality, Sobolev's inequality and \eqref{A}, we can complete the statement.
\end{proof}
\begin{lemma}\label{PS condition}
	$J_{\lambda,T}$ satisfies the $(PS)_{c}$ condition at any level $c\in (0,\frac{1}{3}S^{\frac{3}{2}})$, where $S$ is the best constant of the Sobolev embedding $H^{1}(\mathbb{R}^{3})\hookrightarrow L^{6}(\mathbb{R}^{3})$, i.e.,
$$S=\inf_{u\in D^{1,2}(\mathbb{R}^{3})}\frac{\int_{\mathbb{R}^{3}}|\nabla u|^{2}dx}{(\int_{\mathbb{R}^{3}}|u|^{6}dx)^{\frac{1}{3}}}.$$
\end{lemma}
\begin{proof}
    Let ${u_{n}}$ be a $PS$ sequence satisfying \eqref{Psc}.
    Form Lemma \ref{youjie} we know that $\left\{u_{n}\right\}$ is bounded in $E$, then up to a subsequence, we get
\begin{align}\label{2177}
	\begin{split}
&u_{n}\rightharpoonup u,\ \ \ \ in\ E,\\
&u_{n}\to u,\ \ \ \ in\ L^{q}(\mathbb{R}^{3}),\ 2\le q< 6,\\
&u_{n}\to u,\ \ \ \ a.e.\ in\ \mathbb{R}^{3}.
    \end{split}
\end{align}
Assume $\nu_{n}=u_{n}-u$, From the Brezis-Lieb lemma in \cite{1996Minimax}, we obtain
\begin{align}
	\begin{split}
    &\int_{\mathbb{R}^{3}}|\nabla u_{n}|^{2}dx=\int_{\mathbb{R}^{3}}|\nabla u|^{2}dx+\int_{\mathbb{R}^{3}}|\nabla \nu_{n}|^{2}dx+o(1),\\
    &\int_{\mathbb{R}^{3}}|u_{n}|^{2}dx=\int_{\mathbb{R}^{3}}| u|^{2}dx+\int_{\mathbb{R}^{3}}| \nu_{n}|^{2}dx+o(1),\\
    &\int_{\mathbb{R}^{3}}|u_{n}|^{6}dx=\int_{\mathbb{R}^{3}}| u|^{6}dx+\int_{\mathbb{R}^{3}}|\nu_{n}|^{6}dx+o(1).
    \end{split}
\end{align}
Due to \cite[Theorem A.1]{1996Minimax}, for any $\varphi\in C_{0}^{\infty}\subset E$. we deduce
$$\int_{\mathbb{R}^{3}}h_{T}(u_{n})\varphi dx\to \int_{\mathbb{R}^{3}}h_{T}(u)\varphi dx.$$
It is easy to know
\begin{align}
	\begin{split}
    \int_{\mathbb{R}^{3}}\left(h_{T}(u_{n})u_{n}-h_{T}(u)u\right)dx&=
    \int_{\mathbb{R}^{3}}(h_{T}(u_{n})-h_{T}(u))u_{n}dx
    +\int_{\mathbb{R}^{3}}h_{T}(u)(u_{n}-u)dx\\
    &\le\int_{\mathbb{R}^{3}}(h_{T}(u_{n})-h_{T}(u))u_{n}dx+
    \left(\int_{\mathbb{R}^{3}}(h_{T}(u))^{2}dx\right)^{\frac{1}{2}}|u_{n}-u|_{L^{2}}.
    \end{split}
\end{align}
From the H\"{o}lder inequality, one have
\begin{align}\label{equation 13}
	\begin{split}
		\left|\int_{\mathbb{R}^{3}}\left( \phi_{u_{n}} u_{n}^{2}- \phi_{u}u^{2}\right)dx \right|&\le\int_{\mathbb{R}^{3}}|\phi_{u_{n}}||u_{n}+u||u_{n}-u|dx+
\int_{\mathbb{R}^{3}}|\phi_{u_{n}}-\phi_{u}||u|^{2}dx\\ &\le|\phi_{u_{n}}|_{L^{6}(\mathbb{R}^{3})}|u_{n}+u|_{L^{2}(\mathbb{R}^{3})}
|u_{n}-u|_{L^{3}(\mathbb{R}^{3})}|\\&\ \ \ +|\phi_{u_{n}}-\phi_{u}|_{L^{6}(\mathbb{R}^{3})}
|u|_{L^{\frac{12}{5}}(\mathbb{R}^{3})}^{2}.
	\end{split}
\end{align}
and
\begin{align}\label{equation 14}
	\begin{split}
		\left|\int_{\mathbb{R}^{3}}\left( \phi_{u_{n}}^{2} u_{n}^{2}-\phi_{u}^{2}u^{2}\right) dx\right|&\le\int_{\mathbb{R}^{3}}|\phi_{u_{n}}|^{2}|u_{n}+u||u_{n}-u|dx+\int_{\mathbb{R}^{3}} |\phi_{u_{n}}+\phi_{u}||\phi_{u_{n}}-\phi_{n}||u|^{2}dx\\		&\le|\phi_{u_{n}}|_{L^{6}(\mathbb{R}^{3})}^{2}|u_{n}+u|_{L^{3}(\mathbb{R}^{3})}
|u_{n}-u|_{L^{3}(\mathbb{R}^{3})}|\\&\ \ \ +|\phi_{u_{n}}-\phi_{u}|_{L^{6}(\mathbb{R}^{3})}
|\phi_{u_{n}}+\phi_{u}|_{L^{6}(\mathbb{R}^{3})}|u|_{L^{3}(\mathbb{R}^{3})}^{2}.
	\end{split}
\end{align}
Combine \eqref{2177}-\eqref{equation 14} with Lemma \eqref{FG}, up to subsequence, we get
\begin{align}
	\begin{split}
    &\langle J_{\lambda,T}'(u_{n}),u_{n} \rangle-\langle J_{\lambda,T}'(u),u \rangle\\&\ \
    =\int_{\mathbb{R}^{3}}\left(|\nabla u_{n}|^{2}-|\nabla u|^{2}\right)dx+ \int_{\mathbb{R}^{3}}V(x)(u_{n}^{2}-u^{2})dx
    -\int_{\mathbb{R}^{3}}2\omega(\phi_{u_{n}}u_{n}^{2}-\phi_{u}u^{2})dx\\&\ \ \ \ \ -\int_{\mathbb{R}^{3}}\left(\phi_{u_{n}}^{2}u_{n}^{2}-\phi_{u}^{2}u^{2}\right)dx
    -\int_{\mathbb{R}^{3}}\lambda(h_{T}(u_{n})u_{n}-h_{T}(u)u)dx
    -\int_{\mathbb{R}^{3}}\left(|u_{n}|^{6}-|u|^{6}\right)dx\\&\ \
    =\int_{\mathbb{R}^{3}}\left| \nabla \nu_{n}\right|^{2}dx+V(x)
    \int_{\mathbb{R}^{3}}\nu_{n}^{2}dx-\int_{\mathbb{R}^{3}}\left| \nu_{n}\right| ^{6}dx+o(1),\ \ as\ n\to\infty.
    \end{split}
\end{align}
It is easy to know that $\langle J_{\lambda,T}'(u_{n}),u_{n}\rangle\to \langle J_{\lambda,T}'(u),u\rangle = 0$, we assume that
$$\int_{\mathbb{R}^{3}}\left| \nabla \nu_{n}\right| ^{2}dx+V(x)\int_{\mathbb{R}^{3}}\nu_{n}^{2}dx\to b ,\ \ \ \int_{\mathbb{R}^{3}}\left| \nu_{n}\right| ^{6}dx\to b,$$
where $b$ is nonnegative constant.

We assert that $b=0$. If $b\ne0$, under the definition of $S$ we get
$$\int_{\mathbb{R}^{3}}\left| \nabla \nu_{n}\right| ^{2}dx\ge S\left(\int_{\mathbb{R}^{3}}|\nu_{n}|^{2^{*}}dx\right)^{\frac{1}{3}}.$$
Then
$$\int_{\mathbb{R}^{3}}\left| \nabla \nu_{n}\right| ^{2}dx+V(x)\int_{\mathbb{R}^{3}}\nu_{n}^{2}dx\ge S\left(\int_{\mathbb{R}^{3}}|\nu_{n}|^{2^{*}}dx\right)^{\frac{1}{3}}.$$
which means $b\ge Sb^{\frac{1}{3}}$. Thus $b\ge S^{\frac{3}{2}}$.

As discussed above and follows from $\phi_{u}\le 0$ and $b\ge S^{\frac{3}{2}}$, we can know
\begin{equation*}
	\begin{split}
		c&=\lim_{n\to\infty}J_{\lambda,T}(u_{n})\\
		&\ge\lim_{n\to\infty}\left\{\frac{1}{2}\int_{\mathbb{R}^{3}}\left( \left| \nabla\nu_{n}\right| ^{2}+V(x)\nu_{n}^{2}\right) dx-\frac{1}{6}\int_{\mathbb{R}^{3}}|\nu_{n}|^{6}dx\right\}\\
		&=\frac{1}{3}b\ge\frac{1}{3}S^{\frac{3}{2}}.	
	\end{split}
\end{equation*}
Which have a contradition. Hence $b=0$. Thus
\begin{equation*}
	\begin{split}
		0&\le \left\|\nu_{n}\right\|=\left[\int_{\mathbb{R}^{3}}\left( \left| \nabla\nu_{n}\right| ^{2}+V(x)\nu_{n}^{2}\right) dx\right]^{\frac{1}{2}}\to 0.
	\end{split}
\end{equation*}
\end{proof}
\begin{lemma}\label{}
    $c_{\lambda,T}<\frac{1}{3}S^{\frac{3}{2}},$ where $c_{\lambda,T}$ and $S$ are respectively defined in \eqref{Psc} and Lemma \ref{PS condition}.
\end{lemma}
\begin{proof}
Let $\varphi\in C_{0}^{\infty}$ is a cut-off function satisfying that there exists $R>0$ such taht $\varphi\mid_{B_{R}}=1$, $0 \le\varphi\le 1$ in $B_{2R}$ and $\text{supp}\varphi\subset B_{2R}$. Let $\varepsilon >0$ and define $u_{\varepsilon}:=w_{\varepsilon}\varphi$ where $w_{\varepsilon}\in D^{1,2}(\mathbb{R}^{3})$ is the Talenti function
$w_{\varepsilon}(x)=\frac{(3\varepsilon^{2})^{\frac{1}{4}}}{(\varepsilon^{2}+|x|^{2})^{\frac{1}{2}}}.$
From estimates obtained in \cite{1996Minimax} we get if $\varepsilon$ is small enough,
\begin{align}\label{one}
	\int_{\mathbb{R}^{3}}|\nabla u_{\varepsilon}|^{2}dx=S^{\frac{3}{2}}+O(\varepsilon),
\end{align}
\begin{align}\label{two}
	\int_{\mathbb{R}^{3}}| u_{\varepsilon}|^{6}dx=S^{\frac{3}{2}}+O(\varepsilon^{3}),
\end{align}
\begin{eqnarray}\label{four}
\int_{\mathbb{R}^{3}}|u_{\varepsilon}|^{q}dx
=\left\{   \begin{array}{ll}
	O(\varepsilon^{\frac{q}{2}}),\ &q\in[2,3), \\
    O(\varepsilon^{\frac{q}{2}}|\ln\varepsilon|),\ & q=3,\\
    O(\varepsilon^{\frac{6-q}{2}}),\ & q\in(3,6).
    \end{array}
\right.	
\end{eqnarray}
Since for any $\varepsilon >0$, $\lim_{t\to\infty}J_{\lambda,T}(tu_{\varepsilon})=-\infty$. We can assume there exists $t_{\varepsilon}\ge 0$ such that $\sup_{t\ge 0}J_{\lambda,T}(tu_{\varepsilon})=J_{\lambda,T}(t_{\varepsilon}u_{\varepsilon})$
and without loss of generality we let $t_{\varepsilon}\ge C_{0}>0.$ In fact, suppose there exists a sequence ${\varepsilon_{n}}\subset\mathbb{R}^{+}$ such that $\lim_{n\to\infty}t_{\varepsilon_{n}}=0$ and $J_{\lambda,T}(t_{\varepsilon_{n}}u_{\varepsilon_{n}})=\sup_{t\ge 0}J_{\lambda,T}(tu_{\varepsilon}).$ We can deduce $0<\alpha<c_{\lambda,T}\le\lim_{n\to\infty}J_{\lambda,T}(t_{\varepsilon_{n}}u_{\varepsilon_{n}})=0$
which have a contradiction.

Moreover, we claim that $\left\{t_{\varepsilon}\right\}_{\varepsilon>0}$ is bounded from above. Otherwise, there exists a subsequence ${t_{\varepsilon_{n}}}$ such that $\lim_{n\to\infty}t_{\varepsilon_{n}}=+\infty.$ From \eqref{277}, \eqref{213}, \eqref{one}-\eqref{four} and Lemma \eqref{L 1} we get
\begin{align*}
\begin{split}
0&<c_{\lambda,T}\\&
\le J_{\lambda,T}(t_{\varepsilon_{n}}u_{\varepsilon_{n}})\\&
\le \frac{t_{\varepsilon_{n}}^{2}}{2}\int_{\mathbb{R}^{3}}\left(|\nabla u_{\varepsilon_{n}}|^{2}+V(x)u_{\varepsilon_{n}}^{2}\right)dx
-t_{\varepsilon_{n}}^{2}
\int_{\mathbb{R}^{3}}\omega\phi_{t_{\varepsilon_{n}}}u_{\varepsilon_{n}}^{2}dx\\&\ \ \ -\int_{\mathbb{R}^{3}}\lambda H_{T}(t_{\varepsilon_{n}}u_{\varepsilon_{n}})dx-\frac{t_{\varepsilon_{n}}^{6}}{6}
\int_{\mathbb{R}^{3}}|u_{\varepsilon_{n}}|^{6}dx\\
&\le \frac{t_{\varepsilon_{n}}^{2}}{2}\int_{\mathbb{R}^{3}}\left(|\nabla u_{\varepsilon_{n}}|^{2}+V(x)u_{\varepsilon_{n}}^{2}\right)dx
-t_{\varepsilon_{n}}^{2}\int_{\mathbb{R}^{3}}\omega\phi_{t_{\varepsilon_{n}}}
u_{\varepsilon_{n}}^{2}dx\\&\ \ \ +t_{\varepsilon_{n}}^{2}\int_{\mathbb{R}^{3}}C_{M}u_{\varepsilon_{n}}^{2}dx
-t_{\varepsilon_{n}}^{4}\int_{\mathbb{R}^{3}}\lambda Mu_{\varepsilon_{n}}^{4}dx
-\frac{t_{\varepsilon_{n}}^{6}}{6}
\int_{\mathbb{R}^{3}}|u_{\varepsilon_{n}}|^{6}dx
\\&\le C_{1}t_{\varepsilon_{n}}^{2}-C_{2}t_{\varepsilon_{n}}^{4}-C_{3}t_{\varepsilon_{n}}^{6}\to -\infty,\ \ \ as\ n\to\infty.
\end{split}
\end{align*}
Therefore $0<-\infty$ is a contradiction.

Let
$$\varrho(t)=\frac{t^{2}}{2}\int_{\mathbb{R}^{3}}|\nabla u_{\varepsilon}|^{2}dx-\frac{t^{6}}{6}\int_{\mathbb{R}^{3}}|u_{\varepsilon}|^{6}dx.$$
It is easy to know
\begin{align}
\sup_{t\ge 0}\varrho(t)=\frac{1}{3}S^{\frac{3}{2}}+O(\varepsilon).
\end{align}
According to assumption $(V_{2})$, for $|x|<r$ there exists $\xi>0$ such that
\begin{align}\label{2288}
|V(x)|\le\xi.
\end{align}
Form \eqref{213} and \eqref{one}-\eqref{2288} we obtain
\begin{align*}
    \begin{split}
    J_{\lambda,T}(t_{\varepsilon}u_{\varepsilon})&=\frac{t_{\varepsilon}^{2}}{2}
    \int_{\mathbb{R}^{3}}|\nabla u_{\varepsilon}|^{2}dx+\frac{t_{\varepsilon}^{2}}{2}
    \int_{\mathbb{R}^{3}}V(x)u_{\varepsilon}^{2}dx-\frac{3t_{\varepsilon}^{2}}{4}
    \int_{\mathbb{R}^{3}}\omega\phi_{t_{\varepsilon}u_{\varepsilon}}u_{\varepsilon}^{2}dx
    -\frac{t_{\varepsilon}^{2}}{4}\int_{\mathbb{R}^{3}}\phi_{t_{\varepsilon}u_{\varepsilon}}^{2}
    u_{\varepsilon}^{2}dx\\
    &\ \ \ -\frac{1}{16\pi}\int_{\mathbb{R}^{3}}|\nabla \phi_{t_{\varepsilon}u_{\varepsilon}}|^{2}dx-\int_{\mathbb{R}^{3}}\lambda H_{T}(t_{\varepsilon}u_{\varepsilon})dx
    -\frac{t_{\varepsilon}^{6}}{6}\int_{\mathbb{R}^{3}}|u_{\varepsilon}|^{6}dx\\
    &\le \frac{t_{\varepsilon}^{2}}{2}\int_{\mathbb{R}^{3}}|\nabla u_{\varepsilon}|^{2}dx+\frac{t_{\varepsilon}^{2}}{2}
    \int_{\mathbb{R}^{3}}V(x)u_{\varepsilon}^{2}dx-\frac{3t_{\varepsilon}^{2}}{4}
    \int_{\mathbb{R}^{3}}\omega\phi_{t_{\varepsilon}u_{\varepsilon}}u_{\varepsilon}^{2}dx
    -\lambda Mt_{\varepsilon}^{4}\int_{\mathbb{R}^{3}}u_{\varepsilon}^{4}dx\\
    &\ \ \ +\lambda C_{M}t_{\varepsilon}^{2}\int_{\mathbb{R}^{3}}u_{\varepsilon}^{2}dx
    -\frac{t_{\varepsilon}^{6}}{6}\int_{\mathbb{R}^{3}}|u_{\varepsilon}|^{6}dx\\
    &\le \sup_{t\ge0}\varrho(t)+\frac{t_{\varepsilon}^{2}}{2}\xi\int_{\mathbb{R}^{3}}u_{\varepsilon}^{2}dx
    -\frac{3t_{\varepsilon}^{2}}{4}
    \int_{\mathbb{R}^{3}}\omega\phi_{t_{\varepsilon}u_{\varepsilon}}u_{\varepsilon}^{2}dx
    -\lambda Mt_{\varepsilon}^{4}\int_{\mathbb{R}^{3}}u_{\varepsilon}^{4}dx
    +\lambda C_{M}t_{\varepsilon}^{2}\int_{\mathbb{R}^{3}}u_{\varepsilon}^{2}dx\\
    &\le\sup_{t\ge 0}\varrho(t)+C|u_{\varepsilon}|_{2}^{2}+C|u_{\varepsilon}|_{\frac{12}{5}}^{2}-\lambda MC|u_{\varepsilon}|_{4}^{4}\\
    &\le \frac{1}{3}S^{\frac{3}{2}}+CO(\varepsilon)-\lambda MCO(\varepsilon).
    \end{split}
\end{align*}
When $M$ large enough, we know $CO(\varepsilon)-\lambda MCO(\varepsilon)\to-\infty$ for small enough $\varepsilon>0.$
\end{proof}
\begin{theorem}\label{T2}
Assume $\lambda>0\ T>0$, problem \eqref{new problem} has a nontrivial solution $u_{\lambda,T}$ with $J_{\lambda,T}(u_{\lambda,T})=c_{\lambda,T}.$
\end{theorem}
\begin{proof}
Fist of all, we know that the function $J_{\lambda,T}$ satisfies Lemma \ref{geometric}, that is, the geometric structure of the mountain pass. Then the $PS$ sequence can be obtained. Secondly, because of Lemma \ref{PS condition}, it can be know that function $J_{\lambda,T}$ satisfies the $PS$ condition. According to the mountain pass theorem, there exists a critical point $u_{\lambda,T}\in E$. Moreover, $J_{\lambda,T}=c_{\lambda,T}\ge\alpha>0=J(0)$, so that $u_{\lambda,T}$ is a nontrivial solution.
\end{proof}
\section{Proof of the Theorem 1.1}
In this section, we will prove Theorem \ref{theorem 1.1}. First, prove the solution of problem \ref{new problem} satisfied $|u_{\lambda,T}|_{\infty}\le T$, which means that the solution at this time is the solution of problem \ref{problem 1}. The proof method is similar to the document \cite{MR4143620,MR3950621} using the Nash-Moser method.
\begin{lemma}\label{345}
If $u$ is a critical point of $J_{\lambda,T}$, then $u\in L^{\infty}(\mathbb{R}^{3})$ and
$$|u|_{\infty}\le C_{0}^{\frac{1}{2(\zeta-1)}}\zeta^{\frac{\zeta}{(\zeta-1)^{2}}}\left[(\lambda C_{T}^{*}+\alpha(\varepsilon,u))(1+|u|_{2})^{2}+\lambda C_{T}|u|_{6}^{p-2}\right]^{\frac{1}{2(\zeta-1)}}|u|_{6}^{\kappa},$$
where $C_{0}>0$ and $\kappa\le 1$ are constants independent of $\lambda$ and $T$, $\zeta=\frac{8-p}{2}.$
\end{lemma}
\begin{proof}
Assume $A_{k}=\left\{x\in\mathbb{R}^{3}:|u|^{s-1}\le k\right\},\ B_{k}=\mathbb{R}^{3}\setminus A_{k},$ where $s>1,\ k>0$. Let
\begin{eqnarray}
u_{k}=\left\{   \begin{array}{ll}
	u|u|^{2(s-1)},\ &x\in A_{k}, \\
    k^{2}u,\ &x\in B_{k},
    \end{array}
\right.	
\end{eqnarray}
and
\begin{eqnarray}
\chi_{k}=\left\{   \begin{array}{ll}
	u|u|^{s-1},\ &x\in A_{k}, \\
    ku,\ &x\in B_{k}.
    \end{array}
\right.	
\end{eqnarray}
it is easy to know $|u_{k}| \le |u|^{2s-1}$ if $u_{k},\chi_{k}\in E$, and $\chi_{k}^{2}=uu_{k}\le|u|^{2s}.$ Through direct calculation, the following formula can be obtained:
\begin{eqnarray}
\nabla u_{k}=\left\{   \begin{array}{ll}
	(2s-1)|u|^{2s-2}\nabla u,\ &x\in A_{k}, \\
    k^{2}\nabla u,\ &x\in B_{k},
    \end{array}
\right.	
\end{eqnarray}
\begin{eqnarray}
\nabla \chi_{k}=\left\{   \begin{array}{ll}
	s|u|^{s-1}\nabla u,\ &x\in A_{k}, \\
    k\nabla u,\ &x\in B_{k},
    \end{array}
\right.	
\end{eqnarray}
and
\begin{align}\label{23}
\int_{\mathbb{R}^{3}}\left(|\nabla \chi_{k}|^{2}-\nabla u\nabla u_{k}\right)dx
=(s-1)^{2}\int_{A_{k}}|u|^{2(s-1)}|\nabla u|^{2}dx.
\end{align}
Due to its definition, we get
\begin{align}\label{33}
    \begin{split}
&\int_{\mathbb{R}^{3}}\nabla u\nabla u_{k}dx\\
&\ \ \ =(2s-1)\int_{A_{k}}|u|^{2(s-1)}|\nabla u|^{2}dx+k^{2}\int_{B_{k}}|\nabla u|^{2}dx\\
&\ \ \ \ge (2s-1)\int_{A_{k}}|u|^{2(s-1)}|\nabla u|^{2}dx.
    \end{split}
\end{align}
From \eqref{23} and \eqref{33}, we get $\int_{\mathbb{R}^{3}}\nabla u\nabla u_{k}dx\ge 0$ and
\begin{align}\label{37}
\int_{\mathbb{R}^{3}}|\nabla \chi_{k}|^{2}dx\le s^{2}\int_{\mathbb{R}^{3}}\nabla u\nabla u_{k}dx.
\end{align}
Since $u$ is the critical point, Let $u_{k}$ is a test function in \eqref{qruation 5}, we can deduce
\begin{align}
\int_{\mathbb{R}^{3}}\left(\nabla u\nabla u_{k}+V(x)uu_{k}-2\omega\phi_{u}uu_{k}-\phi_{u}^{2}uu_{k}\right)dx
=\int_{\mathbb{R}^{3}}\lambda h_{T}(u)u_{k}dx+\int_{\mathbb{R}^{3}}|u|^{4}uu_{k}dx.
\end{align}
Combine \eqref{37} with Lemma \ref{L 1} ,it is easy to get
$$\int_{\mathbb{R}^{3}}|\nabla \chi_{k}|^{2}dx\le s^{2}\left(\int_{\mathbb{R}^{3}}\lambda h_{T}(u)u_{k}dx+\int_{\mathbb{R}^{3}}|u|^{4}uu_{k}dx\right).$$
By a version of the Br$\acute{e}$zis-Kato lemma as done in \cite[Lemma 2.5]{MR3529007}, for any $\varepsilon>0$, we can be find $\alpha(\varepsilon,u)$ such that
$$\int_{\mathbb{R}^{3}}|u|^{4}\chi_{k}^{2}dx\le\varepsilon\int_{\mathbb{R}^{3}}|\nabla \chi_{k}|^{2}dx+\alpha(\varepsilon,u)\int_{\mathbb{R}^{3}}|\chi_{k}|^{2}dx.$$
Let $\varepsilon=\frac{1}{2s^{2}}$, from $\chi_{k}^{2}=uu_{k}$ and $(h_{3})$ we deduce
\begin{align}\label{39}
\int_{\mathbb{R}^{3}}|\nabla \chi_{k}|^{2}dx\le 2s^{2}\left(\int_{\mathbb{R}^{3}}\lambda h_{T}(u)u_{k}dx+\alpha(\varepsilon,u)\int_{\mathbb{R}^{3}}|\chi_{k}|^{2}dx\right),
\end{align}
and
\begin{align}\label{310}
|h_{T}(u)u_{k}|\le C_{T}^{*}\chi_{k}^{2}+C_{T}|u|^{p-2}\chi_{k}^{2}.
\end{align}
From the Sobolev embedding theorem, H\"{o}lder inequality, and
\eqref{39}-\eqref{310}, we get
\begin{align}\label{311}
    \begin{split}
    \left(\int_{A_{k}}|\chi_{k}|^{6}dx\right)^{\frac{1}{3}}&\le S^{-1}\int_{\mathbb{R}^{3}}|\nabla \chi_{k}|^{2}dx\\
    &\le S^{-1}2s^{2}\left[\int_{\mathbb{R}^{3}}\lambda\left( C_{T}^{*}\chi_{k}^{2}+C_{T}|u|^{p-2}\chi_{k}^{2}\right)dx
    +\alpha(\varepsilon,u)\int_{\mathbb{R}^{3}}|\chi_{k}|^{2}dx\right]\\
    &\le S^{-1}2s^{2}\left[(\lambda C_{T}^{*}+\alpha(\varepsilon,u))|\chi_{k}|_{2}^{2}+\lambda C_{T}|u|_{6}^{p-2}|\chi_{k}|_{2q}^{2}\right],
    \end{split}
\end{align}
where $q=\frac{6}{8-p}\in (\frac{3}{2},3)$ and $S$ defined in Lemma \ref{PS condition}. Due to $|\chi_{k}|\le |u|^{s}$ and $|\chi_{k}|=|u|^{s}$ for $x\in A_{k}$, together with \eqref{311}, we easily to get
\begin{align}\label{312}
\left(\int_{A_{k}}|u|^{6s}dx\right)^{\frac{1}{3}}\le S^{-1}2s^{2}
\left[\left(\lambda C_{T}^{*}+\alpha(\varepsilon,u)\right)|u|_{2s}^{2s}
+\lambda C_{T}|u|_{6}^{p-2}|u|_{2sq}^{2s}\right].
\end{align}
Through the interpolation inequality, we know $|u|_{2s}\le |u|_{2}^{1-\sigma}|u|_{2qs}^{\sigma}$, where $\sigma\in(0,1)$ and
$\frac{1}{2s}=\frac{1-\sigma}{2}+\frac{\sigma}{2sq}$, so $\sigma=\frac{q(s-1)}{qs-1}$. Moreover since $2s(1-\sigma)=2+\frac{2(1-s)}{qs-1}<2$, we know
\begin{align}\label{313}
|u|_{2s}^{2s}\le|u|_{2}^{2s(1-\sigma)}|u|_{2sq}^{2s\sigma}
\le(1+|u|_{2})^{2}|u|_{2sq}^{2s\sigma}.
\end{align}
When $k\to\infty$, together with \eqref{312}and \eqref{313}, we deduce
\begin{align}\label{314}
    \begin{split}
    |u|_{6s}&\le(S^{-1}2s^{2})^{\frac{1}{2s}}\left[(\lambda C_{T}^{*}+\alpha(\varepsilon,u))(1+|u|_{2})^{2}|u|_{2sq}^{2s\sigma}
    +\lambda C_{T}|u|_{6}^{p-2}|u|_{2sq}^{2s}\right]^{\frac{1}{2s}}\\
    &\le C_{0}^{\frac{1}{2s}}s^{\frac{1}{s}}\left[(\lambda C^{*}_{T}+\alpha(\varepsilon,u))(1+|u|_{2})^{2}+\lambda C_{T}|u|_{6}^{p-2}\right]^{\frac{1}{2s}}|u|_{2sq}^{\kappa},
    \end{split}
\end{align}
where $\kappa\in\left\{\sigma,1\right\}$, $C_{0}=\max\left\{2S^{-1},1\right\}$. Let $\zeta=\frac{6}{2q},$ then $\zeta\in(1,2)$. Now we use $j$ iterations by letting $s_{j}=\zeta^{j}$ in \eqref{314}, then we get
\begin{align}\label{315}
|u|_{6\zeta^{j}}\le C_{0}^{\frac{1}{2}\sum_{j=1}^{\infty}\frac{1}{\zeta^{j}}}
\zeta^{\sum_{j=1}^{\infty}\frac{j}{\zeta^{j}}}\left[(\lambda C_{T}^{*}+\alpha)(1+|u|_{2})^{2}+\lambda C_{T}|u|_{6}^{p-2}
\right]
^{\frac{1}{2}\sum_{j=1}^{\infty}\frac{1}{\zeta^{j}}}|u|_{6}^
{\kappa_{1}\cdot\cdot\cdot\kappa_{j}},
\end{align}
where $\sigma_{j}=\frac{q(\zeta_{j}-1)}{q\zeta^{j}-1}<1$, $\kappa_{j}\in\left\{\sigma_{j},1\right\}\le 1$. From a easy calculation, we get
$$\sum_{j=1}^{\infty}\frac{1}{\zeta^{j}}=\frac{1}{\zeta-1},\ \ \ \ \
\sum_{j=1}^{\infty}\frac{j}{\zeta^{j}}=\frac{\zeta}{(\zeta-1)^{2}}.$$

The estimates of $|u|_{\infty}$ will be divided into two cases.
\begin{enumerate}[(i)]
    \item When $|u|_{6}\ge 1$, $|u|_{6}^{\kappa_{1}\cdot\cdot\cdot\kappa_{j}}\le |u|_{6}$. If $j\to\infty$ in equation \eqref{315}, we can get
        \begin{align}
        |u|_{\infty}\le C_{0}^{\frac{1}{2(\zeta-1)}}\zeta^{\frac{\zeta}{(\zeta-1)^{2}}}\left[(\lambda C_{T}^{*}+\alpha(\varepsilon,u))(1+|u|_{2})^{2}+\lambda C_{T}|u|_{6}^{p-2}\right]^{\frac{1}{2(\zeta-1)}}|u|_{6}.
        \end{align}
    \item When $|u|_{6}<1$, by $\sigma_{j}=\frac{q(\zeta^{j}-1)}{q\zeta^{j}-1}\ge 1-\frac{1}{\zeta^{j}}$ and $\kappa_{j}\in\left\{\sigma_{j},1\right\}$, for any $j\in\mathbb{N}$, we deduce
        $$0<\sigma_{1}\sigma_{2}\cdot\cdot\cdot\sigma_{j}\le\kappa_{1}
        \kappa_{2}\cdot\cdot\cdot\kappa_{j}.$$
        For $s\in(0,1)$, we claim that $\ln(1-s)\ge -s-\frac{s^{2}}{2(1-s)^{2}}$. Then we can easy to get
        $$\sum_{i=1}^{j}\ln\kappa_{i}\ge\sum_{i=1}^{j}\ln\sigma_{i}\ge -\sum_{i=1}^{j}\frac{1}{\zeta^{i}}
        -\frac{1}{2}\sum_{i=1}^{j}\frac{1}{(\zeta^{i}-1)^{2}}.$$
        From direct calculation, we get
        $$\sum_{i=1}^{j}\frac{1}
        {(\zeta^{i}-1)^{2}}\le\frac{\zeta^{2}}{(\zeta^{2}-1)(\zeta-1)^{2}}.$$
        So
        $$\sum_{i=1}^{\infty}\ln\kappa_{i}\ge-\frac{1}{\zeta-1}
        -\frac{\zeta^{2}}{2(\zeta^{2}-1)(\zeta-1)^{2}}:=\theta$$
        Therefore, $\kappa_{1}\kappa_{2}\cdot\cdot\cdot\kappa_{j}\ge e^{\theta}$ for any $j\in\mathbb{N}$. Due to $|u|_{6}<1$ then $|u|_{6}^{\kappa_{1}\kappa_{2}\cdot\cdot\cdot\kappa_{j}}\le |u|_{6}^{e^{\theta}}$.If $j\to\infty$ in equation \eqref{315}, we can get
        \begin{align}
        |u|_{\infty}\le C_{0}^{\frac{1}{2(\zeta-1)}}\zeta^{\frac{\zeta}{(\zeta-1)^{2}}}\left[(\lambda C_{T}^{*}+\alpha(\varepsilon,u))(1+|u|_{2})^{2}+\lambda C_{T}|u|_{6}^{p-2}\right]^{\frac{1}{2(\zeta-1)}}|u|_{6}^{e^{\theta}}.
        \end{align}
\end{enumerate}
Combining the two cases, we know that the proof is complete if $\kappa=1$ or $\kappa=e^{\theta}\le 1.$
\end{proof}
\begin{proof}[\textbf{Proof of Theorem 1.1}]
Let $u\in C_{0}^{\infty}(\mathbb{R}^{3})$ and $u(x)\le 0$, then from the definition of equation $h_{T}(t)$ in \eqref{a1}, we know $H_{T}(tu)=0$ for any $t>0$. Hence,
\begin{align}
    \begin{split}
    J_{\lambda,T}(tu)=&
    \frac{t^{2}}{2}\int_{\mathbb{R}^{3}}\left(\left | \nabla u\right |^{2}+V(x)u^{2}\right)dx-t^{2}\int_{\mathbb{R}^{3}}\omega\phi_{tu}u^{2}dx-
\frac{t^{2}}{2}\int_{\mathbb{R}^{3}}\phi_{tu}^{2}u^{2}-
\frac{1}{8\pi}\int_{\mathbb{R}^{3}}|\nabla \phi_{tu}|^{2}dx\\
&\ \ \ -\frac{\beta}{16\pi}\int_{\mathbb{R}^{3}}|\nabla \phi_{tu}|^{4}dx-\frac{t^{6}}{6}\int_{\mathbb{R}^{3}}|u|^{6}dx.
    \end{split}
\end{align}
it shows $J_{\lambda,T}(tu)\to-\infty$ for $t\to+\infty$. Then it can be find a $t_{0}>0$ such that $J_{\lambda,T}(t_{0}u)<0.$ we can assume $\gamma(\cdot)=tt_{0}u,t\in [0,1]$, so $\gamma(t)\in\Gamma$. Because $H_{T}(u)=0$, as $t\in[0,1]$, we get
\begin{align}
    \begin{split}
    c_{\lambda,T}&\le\max_{t\in [0,1]}J_{\lambda,T}(\gamma(t))\\
    &\le\max_{t\ge 0}\left\{\frac{t^{2}}{2}\int_{\mathbb{R}^{3}}\left(\left | \nabla u\right |^{2}+V(x)u^{2}\right)dx-t^{2}
    \int_{\mathbb{R}^{3}}\omega\phi_{tu}u^{2}dx
    -\frac{t^{6}}{6}\int_{\mathbb{R}^{3}}|u|^{6}dx.\right\}:=D>0,
    \end{split}
\end{align}
where $D$ is constant independent of $\lambda$ and $T$. From Theorem \ref{T2}, $(h_{4})$ and $(V_{1})-(V_{2})$, we know
\begin{align}\label{320}
    \begin{split}
    4D&\ge 4c_{\lambda,T}\\
    &\ge4J_{\lambda,T}-\langle J'_{\lambda,T}(u_{\lambda,T}),u_{\lambda,T}\rangle\\
    &=\int_{\mathbb{R}^{3}}|\nabla u_{\lambda,T}|^{2}+V(x)u_{\lambda,T}^{2}dx+\frac{\beta}{4\pi}\int_{\mathbb{R}^{3}}
    |\nabla\phi_{u_{\lambda,T}}|^{4}dx+\int_{\mathbb{R}^{3}}\phi_{u_{\lambda,T}}^{2}
    u_{\lambda,T}^{2}dx
    \\&\ \ \ +\int_{\mathbb{R}^{3}}\lambda\left(h_{T}(u_{\lambda,T})
    u_{\lambda,T}-4H_{T}(u_{\lambda,T})\right)dx+\frac{1}{3}\int_{\mathbb{R}^{3}}
    |u_{\lambda,T}|^{6}dx
    \\&\ge
    \int_{\mathbb{R}^{3}}|\nabla u_{\lambda,T}|^{2}+V(x)u_{\lambda,T}^{2}dx-\lambda\mu\int_{\mathbb{R}^{3}}
    u_{\lambda,T}^{2}dx
    \\&\ge\frac{1}{2}
    \left\|u_{\lambda,T}\right\|^{2}+(\frac{V_{0}}{2}-\lambda\mu)|u_{\lambda,T}|_{2}^{2}.
    \end{split}
\end{align}
We can find a $\lambda_{0}$ such that $\frac{V_{0}}{2}-\lambda_{0}\mu>0$. Then from \eqref{320}, $\left\|u_{\lambda,T}\right\|\le 8D.$ Hence, we deduce
$$|u_{\lambda,T}|_{2}\le C_{5},\ \ \ |u_{\lambda,T}|_{6}^{2}\le C_{6},$$
where $C_{5},C_{6}>0$ independent of $\lambda,T.$

From Lemma \eqref{345}, we get
$$|u_{\lambda,T}|_{\infty}\le C_{0}^{\frac{1}{2(\zeta-1)}}\zeta^{\frac{\zeta}{(\zeta-1)^{2}}}\left[(\lambda C_{T}^{*}+\alpha(\varepsilon,u))(1+C_{5})^{2}+\lambda C_{T}C_{6}^{p-2}\right]^{\frac{1}{2(\zeta-1)}}C_{6}^{\kappa}.$$
So, we can choose $T>0$ large enough such that
$$C_{0}^{\frac{1}{2(\zeta-1)}}\zeta^{\frac{\zeta}{(\zeta-1)^{2}}}
\left[\alpha(\varepsilon,u))(1+C_{5})^{2}\right]^{\frac{1}{2(\zeta-1)}}
C_{6}^{\kappa}\le\frac{T}{2}.$$
Since $C_{T}^{*},$ $C_{T}$ are fixed constants for above $T$, we can choose $\lambda_{1}^*<\lambda_{0}$ such that
$$|u_{\lambda,T}|_{\infty}\le C_{0}^{\frac{1}{2(\zeta-1)}}\zeta^{\frac{\zeta}{(\zeta-1)^{2}}}\left[(\lambda_{1}^* C_{T}^{*}+\alpha(\varepsilon,u))(1+C_{5})^{2}+\lambda_{1}^* C_{T}C_{6}^{p-2}\right]^{\frac{1}{2(\zeta-1)}}C_{6}^{\kappa}\le T.$$
Then, for $\lambda\in(0,\lambda_{1}^*)$, we can get $|u_{\lambda,T}|_{\infty}\le T$, $u_{\lambda,T}$ is also a solution for the problem \eqref{problem 1}.
\end{proof}
\section{Proof of the theorem 1.4 }
\begin{proof}
When $\beta=0$ check the proof of theorem 1.1 we can easy to get theorem \ref{corollary 1.2}, hence we omit the details.
\end{proof}

\bibliography{refss}

\end{document}